\documentclass[12pt, reqno]{amsart}

\usepackage{amsmath,amsfonts,amsbsy,amsgen,amscd,mathrsfs,amssymb,amsthm}
\usepackage[usenames,dvipsnames]{xcolor}

\usepackage[bookmarks=true,
bookmarksnumbered=true, breaklinks=true,linkcolor=blue, citecolor=purple,
pdfstartview=FitH, hyperfigures=false,
plainpages=false, naturalnames=true,
colorlinks=true,pagebackref=true,
pdfpagelabels]{hyperref}
\hypersetup{pdftitle={ },pdfauthor={ }}

\usepackage[a4paper,margin=1in]{geometry}
\usepackage{amsmath}
\allowdisplaybreaks[4]
\numberwithin{equation}{section}
\newtheorem{theorem}{Theorem}[section]

\newtheorem{lemma}[theorem]{Lemma}
\newtheorem{corollary}[theorem]{Corollary}
\newtheorem{remark}[theorem]{Remark}

\begin{document}
\title[The $L_q$-Minkowski problem of anisotropic $p$-torsional rigidity]{The $L_q$-Minkowski problem of anisotropic $p$-torsional rigidity}

\author{Chao Li}
\address{Chao Li:
School of Mathematics and Statistics, Ningxia University, Yinchuan, Ningxia, 750021, China.}
\address{Ningxia Basic Science Research Center of Mathematics, Ningxia University, Yinchuan 750021, China.}
	\email{lichao@nxu.edu.cn, lichao166298@163.com}

\author{Bin Chen}
\address{Bin Chen: Department of Mathematics, Lanzhou University of Technology, Lanzhou, Gansu, 730050, China.} 
\email{chenb121223@163.com}

	\subjclass[2020]{35N25, 52A20, 31A15}
	
\keywords{anisotropic $p$-Laplacian operator, anisotropic $p$-torsional rigidity, $L_q$-Minkowski problem}

\begin{abstract}
In this paper, the $L_q$-Minkowski problem of anisotropic $p$-torsional rigidity is considered. The existence of the solution to the $L_q$-Minkowski problem for anisotropic $p$-torsional rigidity with $0<q<1$ and $1<q\neq n+\frac{p}{p-1}$ is given.
\end{abstract}

\maketitle

\vskip 20pt
\begin{sloppypar}
\section{Introduction and main results}

The Minkowski problem has long been a cornerstone of convex geometry, occupying a central role in fields such as nonlinear partial differential equations, differential geometry and optimal transport theory. The classical Minkowski problem is primarily concerned with the existence, uniqueness and regularity of solutions, to which seminal contributions have been made by scholars such as Aleksandrov, Fenchel and Jessen (see the Notes for Section 8.2 in \cite{SRA2014}). Since Lutwak \cite{LUTWA1993} introduced and studied the $L_q$-Minkowski problem, it has become evident that varying the parameter $q$ not only generalizes the classical formulation but also introduces substantial complexity in the analysis of uniqueness, existence and regularity when $q<1$. This development has attracted considerable attention and inspired a series of notable research advances.

In the study of the classical Minkowski problem, the various generalized formulations have attracted considerable academic interest and demonstrated substantial research value. Notable examples include the Orlicz-Minkowski problem introduced by Haberl, Lutwak, Yang and Zhang \cite{HLYZ2010}; the dual Minkowski problem investigated by Huang, Lutwak, Yang and Zhang \cite{HLYZ2016}; and the chord Minkowski problem explored by Lutwak, Xi, Yang and Zhang \cite{LXYZ2024}. Other important extensions involve the Minkowski problems related to electrostatic capacity \cite{JER1996, ZDX2020}, Gaussian measure \cite{HXZ2021, LJQ2022, FLX2025}, hyperbolic space \cite{LX2022}, the Gauss image measure \cite{BLYZ2020}, weighted settings \cite{KLL2023, LWG2023, LIV2019}, harmonic measure \cite{AKM2023, JER1989, JER1991, LCZX2024} and coconvex sets of finite volume \cite{SRA2018, LNYZ2024}.

Let $\mathscr{K}^n$ denote the set of convex bodies (compact, convex subsets with non-empty interiors) in the Euclidean space $\mathbb{R}^n$. Let $\mathscr{K}^n_o$ denote the set of convex bodies containing the origin in their interiors in $\mathbb{R}^n$.  Besides, let $\mathbb{S}^{n-1}$ denote the unit sphere, $\mathbb{B}^{n}$ denote the unit ball and $V(K)$ denote the $n$-dimensional volume of $K$.
In this paper, we further investigate a novel form of the Minkowski problem, namely the $L_q$-Minkowski problem of anisotropic $p$-torsional rigidity. This line of research follows the seminal works of Jerison \cite{JER1991} on the Minkowski problem for electrostatic capacity, Colesanti and Fimiani \cite{CAM2010} on the Minkowski problem for torsional rigidity, more recently, Langharst and Ulivelli \cite{LAU2024} on the log-Minkowski problem for torsional rigidity. Let $n \ge 2$ and $1 < p < \infty$ and let $F: \mathbb{R}^n \to \mathbb{R}$ be a positive, convex and $1$-homogeneous function. For a bounded open set $K \subset \mathbb{R}^n$, denote by $u \in W_0^{1,p}(K)$ the unique solution to the following boundary value problem (see \cite{PDG2014})
\begin{align*}
\left\{\begin{array}{lll}
\Delta_p^F u=-1 & \text { in } & K, \\
u=0 & \text { on } & \partial K,
\end{array}\right.
\end{align*}
where $\Delta_p^F$ is the anisotropic $p$-Laplacian operator given by
\begin{align}\label{Eq:APYZ}
\Delta_p^F u=\operatorname{div}\left(F^{p-1}(\nabla u) \nabla_{\xi} F(\nabla u)\right).
\end{align}
The anisotropic $p$-torsional rigidity $\tau_{F,p}(K) > 0$ of $K$ is defined as (see \cite{PDG2014})
\begin{align}\label{EQ:YXLG}
\tau_{F,p}(K)=\int_{K} F^p\left(\nabla u\right) d x=\int_{K} u d x.
\end{align}
From (\ref{EQ:YXLG}) and anisotropic Poho\v{z}aev identity \cite{BCG2018}, Li \cite{LCHAO2025} derived the following integral expression
\begin{gather*}
\tau_{F,p}(K)=\frac{p-1}{n(p-1)+p} \int_{\mathbb{S}^{n-1}} h_K(\xi) F^p(\nabla u(\mathbf{g}^{-1}_K(\xi)))d S(K,\xi),
\end{gather*}
where $h_K$ denotes the support function of  $K$, $\mathbf{g}_K: \partial^{\prime} K \rightarrow \mathbb{S}^{n-1}$ is the Gauss map of $K$, defined on $\partial^{\prime} K$ (the set of points of $\partial K$ that have a unique outer unit normal) and $S(K,\cdot)$ denotes the surface area measure of $K$. Subsequently, the anisotropic $p$-torsional measure $S_{F,p}(K,\cdot)$ of $K$ is defined as a Borel measure on the unit sphere $\mathbb{S}^{n-1}$ such that for any Borel set $\eta \subset \mathbb{S}^{n-1}$, it is given by
\begin{align*}
S_{F,p}(K,\eta)=\int_{x\in\mathbf{g}_K^{-1}(\eta)}F^p\left(\nabla u(x)\right) d\mathscr{H}^{n-1}(x),
\end{align*}
where $u \in W_0^{1, p}(K) \backslash\{0\}$ and $\mathscr{H}^{n-1}$ is $(n-1)$-dimensional Hausdorff measure. In recent years, substantial progress has been achieved in the study of the anisotropic $p$-torsional rigidity and the anisotropic $p$-Laplacian operator. For further details, we refer the reader to References \cite{BCG2018, PDG2014, FNS2018, GRAM024}.

Motivated by Lutwak's \cite{LUTWA1993} variational formula of the $L_q$-Minkowski problem, we naturally turn to its analogue in the setting of anisotropic $p$-torsional rigidity. In this paper, we establish a proof for the following variational formula (see Section \ref{sec22})
$$
\lim _{t \rightarrow 0} \frac{\tau_{F,p}\left([h_t]\right)-\tau_{F,p}(K)}{t}=\frac{1}{q}\int_{\mathbb{S}^{n-1}} f^q(v) d S_{F,p,q}(K,v),
$$ 
where $h_t(v)=(h_K(v)^q+t f(v)^q)^{\frac{1}{q}}$ for $q\neq 0$. This naturally leads to the definition of $L_q$ anisotropic $p$-torsional measure: For $1<p<\infty$ and $q\neq 0$, the $L_q$ anisotropic $p$-torsional measure  $S_{F,p,q}(K,\cdot)$ of $K$ is a Borel measure on the unit sphere $\mathbb{S}^{n-1}$ defined for a Borel set $\eta\subset \mathbb{S}^{n-1}$ by
\begin{gather*}
S_{F,p,q}(K,\eta)=\frac{p-1}{n(p-1)+p}\int_{x\in\mathbf{g}_K^{-1}(\eta)}\langle x, \mathbf{g}_K(x)\rangle^{1-q} F^p\left(\nabla u(x)\right) d\mathscr{H}^{n-1}(x).
\end{gather*} 

The $L_q$-Minkowski problem of the anisotropic $p$-torsional rigidity can be stated as follows:

\noindent{\bf The $L_q$-Minkowski problem of the anisotropic $p$-torsional rigidity:}~~{\it 
 Let $1<p<\infty$, $q\neq 0$ and $\mu$ be a Borel measure on $\mathbb{S}^{n-1}$. What are the necessary and sufficient conditions  on $\mu$, so that there exists a convex body $K \in \mathscr{K}_o^n$ such that
\begin{align}\label{Eq:LQAMW}
S_{F,p,q}(K, \cdot)=\mu ?
\end{align}}

Within the framework of the anisotropic $p$-Laplacian operator (\ref{Eq:APYZ}), this formulation unifies and generalizes several fundamental differential operators: specifically, it encompasses the anisotropic Laplacian, the isotropic $p$-Laplacian and the classical (i.e., linear) Laplacian as special cases. Regarding the $L_q$-Minkowski problem for anisotropic $p$-torsional rigidity (\ref{Eq:LQAMW}), the case $q = 1$ corresponds to the Minkowski problem for anisotropic $p$-torsional rigidity which has been studied by Li \cite{LCHAO2025}. When $q=0$, the problem transforms into the log-Minkowski problem of the anisotropic $p$-torsional rigidity \cite{LCHAO2025}. Moreover, when $F(\xi) = \sum_k|\xi_k|$, is exactly the $L_q$-Minkowski problem of the $p$-torsional rigidity (see \cite{CZZW2024,SXZ2021}). Meanehile, when $F(\xi) = \sum_k|\xi_k|$ and $p=2$, we encounter the $L_q$-Minkowski problem of the torsional rigidity (see \cite{CHD2020,HJR2021}). If $q=1$, $p=2$ and $F(\xi) = \sum_k|\xi_k|$, the problem simplifies to the most classical Minkowski problem of torsional rigidity \cite{COA2005,CAM2010}. In the specific case where $q=0$, $p=2$ and $F(\xi) = \sum_k|\xi_k|$, the log-Minkowski problem of torsional rigidity was investigated by Langharst and Ulivelli \cite{LAU2024} for the origin-symmetric convex bodies. Subsequently, Hu \cite{HJR20241} extended the result of Langharst and Ulivelli \cite{LAU2024} to the setting without symmetry assumptions. Recently, Zhao and Zhao \cite{ZXZP20261,ZXZP2026} researched the Orlicz–Minkowski problem for the $q$-torsional rigidity  and the $p$-th dual  Minkowski  problem  for the $k$-torsional rigidity corresponding to a $k$-Hessian equation via a curvature flow approach. For further exploration of the research on the Minkowski problem of the torsional rigidity, refer to references \cite{HJR2024,HJRM2022,HSX2018,LNYZ2020,SXZ2021}.

Firstly, the existence of the $L_q$-Minkowski problem of the anisotropic $p$-torsional rigidity with $q>1$ is established.  

\begin{theorem}\label{Thm:pdy1} 
Let $1<p<\infty$, $1<q\neq n+\frac{p}{p-1}$ and $\mu$ be a nonzero finite Borel measure on $\mathbb{S}^{n-1}$ that is not concentrated in any closed hemisphere, then there exists a convex body $K \in \mathscr{K}_o^n$ such that
\begin{align*}
S_{F,p,q}(K, \cdot)=\mu.
\end{align*}
\end{theorem}

\begin{remark}
In the special case $F(\xi) = \sum_k|\xi_k|$ and $p=2$, Theorem \ref{Thm:pdy1} coincides with the findings of Chen and Dai \cite{CHD2020}.
\end{remark}

Moreover, we establish the existence of the $L_q$-Minkowski problem to the anisotropic $p$-torsional rigidity with $0<q<1$ through the application of the polyhedral approximation methods.

\begin{theorem}\label{Thm:pdy0x1} 
Let $1<p<\infty$, $0<q<1$ and $\mu$ be a finite Borel measure on $\mathbb{S}^{n-1}$ which is not concentrated in any closed hemisphere. Then there exists a convex body $K$ in $\mathbb{R}^n$ such that
\begin{equation*}
d\mu=\lambda dS_{F,p,q}(K, \cdot),
\end{equation*}
where $\lambda$ is a positive constant depending on $p$, $q$, $n$ and $\mu$.
\end{theorem}

\begin{remark}
Taking $F(\xi) = \sum_k|\xi_k|$ in Theorem \ref{Thm:pdy0x1}, then it is exactly Chen, Zhao, Wang and Zhao's \cite[Theorem 1.6]{CZZW2024} result and if we add $p=2$ on this basis again, we get the result of Hu and Liu's \cite[Theorem 1]{HJR2021}.
\end{remark}

\section{Notations and Background Materials}\label{SEC2}

\subsection{Convex Geometry}
A convex body $K\in \mathscr{K}^n$ is uniquely determined by its support function $h(K, \cdot)=h_K:\mathbb{R}^n\rightarrow\mathbb{R}$, which is defined by
\begin{align}\label{SF}
h(K, x)=\max\{ x\cdot y : y\in K\}, \ \ \ \ x\in \mathbb{R}^n,
\end{align}
where $x\cdot y$ denotes the standard inner product of $x$ and $y$ in $\mathbb{R}^n$. It is also clear from the definition that $h(K,u) \leq h(L,u)$ for all  $u \in \mathbb{S}^{n-1}$ if and only if $K\subseteq L$.

For $q\neq0$ and $\alpha, \beta\geq0$ (not both zero) and $K,L\in\mathscr{K}^n_o$, the $L_q$-Minkowski combination \cite{LYZ2018} $\alpha\cdot K+_q\beta\cdot L$ of $K$ and $L$ is defined by 
\begin{align}\label{eq:qbd0}
\alpha\cdot K+_q\beta\cdot L=\bigcap_{y\neq 0}\{x\in\mathbb{R}^n:x\cdot y\leq \left(\alpha h_K(u)^q +\beta h_L(u)^q\right)^\frac{1}{q}\}.
\end{align}
For $q\geq1$, the support function of $L_q$-Minkowski combination satisfies (see \cite{SRA2014})
\begin{align}\label{lpmz}
h(\alpha\cdot K+_q\beta\cdot L,\cdot)^q=\alpha h(K,\cdot)^q+\beta h(L,\cdot)^q,
\end{align}
where $``+_q"$ denotes the $L_q$-sum and $\alpha\cdot K=\alpha^{1 / q} K$ is the $L_q$-Minkowski scalar multiplication. If $q=1$, (\ref{lpmz}) is just the classical Minkowski linear combination.

If $K$ is a compact star-shaped set (about the
origin) in $\mathbb{R}^n$, its radial function,
$\rho_K=\rho(K,\cdot):\mathbb{R}^n\setminus\{0\}\rightarrow[0,\infty)$ is defined by
\begin{align}\label{RF}
\rho(K,x)=\max\{c\geq0: c x\in K\}, \ \ x\in\mathbb{R}^n\setminus \{0\}.
\end{align}
If $\rho_K$ is positive and continuous, then $K$ is called a star body (with respect to the origin).

The radial map $r_K: \mathbb{S}^{n-1} \rightarrow \partial K$ is given by
$$
r_K(\xi)=\rho_K(\xi) \xi, \ \ \xi \in \mathbb{S}^{n-1},
$$
i.e., $r_K(\xi)$ is the unique point on $\partial K$ located on the ray parallel to $\xi$ and emanating from the origin.

Using this,  the outer unit normal vector to $\partial K$ at $x$, denoted by $\mathbf{g}(x)$, is well defined for $\mathscr{H}^{n-1}$ almost all $x \in \partial K$. The map $\mathbf{g}: \partial K \rightarrow \mathbb{S}^{n-1}$ is called the Gauss map of $K$. For $\omega \subset \mathbb{S}^{n-1}$, let
$$
\mathbf{g}^{-1}(\omega)=\{x \in \partial K: \mathbf{g}(x) \text { is defined and } \mathbf{g}(x) \in \omega\}.
$$

If $\omega$ is a Borel subset of $\mathbb{S}^{n-1}$, then $\mathbf{g}^{-1}(\omega)$ is $\mathscr{H}^{n-1}$-measurable. The Borel measure $S(K,\cdot)$, on $\mathbb{S}^{n-1}$, is defined for Borel $\omega \subset \mathbb{S}^{n-1}$ by
$$
S(K,\omega)=\mathscr{H}^{n-1}\left(\mathbf{g}^{-1}(\omega)\right)
$$
and is called the surface area measure of $K$. For every $f \in C\left(\mathbb{S}^{n-1}\right)$,
\begin{align}\label{Eq:gdfs}
\int_{\mathbb{S}^{n-1}} f(\xi) dS(K,\xi)=\int_{\partial K} f(\mathbf{g}(x)) d\mathscr{H}^{n-1}(x).
\end{align}

Let $C^+(\mathbb{S}^{n-1})$ denote a positive continuous function on $\mathbb{S}^{n-1}$ and let $E\in\mathbb{S}^{n-1}$  be a closed subset not contained in any closed hemisphere,  if $f\in C^+(\mathbb{S}^{n-1})$, the Wulff shape  (also known as the Aleksandrov body) $[f]$ associated with $f$ is defined as
\begin{align*}
[f]=\left\{x \in \mathbb{R}^{n}: x \cdot u \leq f(u) \text { for all } u \in E\right\}.
\end{align*}
Equivalently, $[f]$ is the unique maximal (with respect to set inclusion) element of
\begin{align}\label{Eq:ele}
\left\{K \in \mathscr{K}^n_{0}: h_{K}(u) \leq f(u), u \in E\right\}.
\end{align}
In fact,
\begin{align*}
[f]=\bigcap_{u \in E} H^{-}(u, f(u)),
\end{align*}
where
\begin{align*}
H^{-}(u, t)=\left\{x \in \mathbb{R}^{n}: x \cdot u \leq t\right\} .
\end{align*}

\begin{lemma}[Proposition 2.5 of \cite{KLL2023}]\label{lem:jxy1}
Let $K \in \mathscr{K}_o^n$, the Jacobian of $r_K: \mathbb{S}^{n-1} \rightarrow \partial K$ is $\frac{\rho_K^n(v)}{h_K(\mathbf{g}_K(r_K(v)))}$ up to a set of measure zero.
\end{lemma}

\begin{lemma}[Lemma 2.1 of \cite{LJQ2022}]\label{lem:Lpbfyl}
 For $q \neq 0$, let $K \in \mathscr{K}_o^n$ and $f: \mathbb{S}^{n-1} \rightarrow \mathbb{R}$ be a continuous function. For small enough $\delta>0$ and each $t \in(-\delta, \delta)$, we define the continuous function $h_t: \mathbb{S}^{n-1} \rightarrow(0, \infty)$ as
$$
h_t(v)=\left(h_K(v)^q+t f(v)^q\right)^{\frac{1}{q}}, \quad v \in \mathbb{S}^{n-1}.
$$
Then 
$$
\lim _{t \rightarrow 0} \frac{\rho_{\left[h_t\right]}(v)-\rho_K(v)}{t}=\frac{f(\mathbf{g}_K(r_K(v)))^q}{q h_K(\mathbf{g}_K(r_K(v)))^q} \rho_K(v),
$$
holds for almost all $u \in \mathbb{S}^{n-1}$.  Moreover, there exists $M>0$ such that
$$
\left|\rho_{\left[h_t\right]}(v)-\rho_K(v)\right| \leq M|t|,
$$
for all $v \in \mathbb{S}^{n-1}$ and $t \in(-\delta, \delta)$.
\end{lemma}

\subsection{Anisotropic $p$-torsional rigidity}\label{sec22}

Let $F: \mathbb{R}^n \to[0,+\infty)$ be a function, such that

(i) $F$ is convex,

(ii) $F(\xi) \geq 0$ for $\xi \in \mathbb{R}^n$ and $F(\xi)=0$ if and only if $\xi=0$,

(iii) $F(t \xi)=|t| F(\xi)$ for $\xi \in \mathbb{R}^n$ and $t \in \mathbb{R}$.

Then $F$ is  a norm on  $\mathbb{R}^n$.  For $1<p<\infty$ and for some $\alpha \in(0,1)$,  define
$$
\mathscr{I}_p=\left\{F \in C^{2, \alpha}\left(\mathbb{R}^n \backslash\{0\}\right), \quad \frac{1}{p} F^p \in C_{+}^2\left(\mathbb{R}^n \backslash\{0\}\right)\right\},
$$
i.e., $F$ is a regular norm, where $F^p$ is a twice continuously differentiable function in $\mathbb{R}^n \backslash\{0\}$ whose Hessian matrix is strictly positive definite.

Let $1<p<\infty$ and $n \geq 2$, for a bounded open set $K \subset \mathbb{R}^n$, denote by $u \in W_0^{1,p}(K)$ the unique solution to the following boundary value problem (see \cite{PDG2014})
\begin{align}\label{Eq:patr}
\left\{\begin{array}{lll}
\Delta_p^F u=-1 & \text { in } & K, \\
u=0 & \text { on } & \partial K,
\end{array}\right.
\end{align}
where $\Delta_p^F$ is the Finsler $p$-Laplacian (also called anisotropic $p$-Laplacian) given by
\begin{align}\label{Eq:fsplls}
\Delta_p^F u=\operatorname{div}\left(F^{p-1}(\nabla u) \nabla_{\xi} F(\nabla u)\right).
\end{align}
Let $u$ be the unique solution of problem (\ref{Eq:patr}) in $K$, then the anisotropic $p$-torsional rigidity $\tau_{F,p}(K)$ of $K$ is defined by  (see \cite{LCHAO2025})
\begin{gather}\label{Eq:AFNZg}
\tau_{F,p}(K)=\frac{p-1}{n(p-1)+p} \int_{\mathbb{S}^{n-1}} h_K(\xi) F^p(\nabla u(\mathbf{g}^{-1}_K(\xi)))d S(K,\xi), 
\end{gather}
if $p=2$ and $F(\xi)=\sum_k\left|\xi_k\right|$ in (\ref{Eq:patr}) and (\ref{Eq:AFNZg}), then $\tau_{F,p}(K)$ becomes the classical torsional rigidity $\tau(K)$ of $K$ (see \cite{COA2005}).

Moreover, Li \cite[Proposition 3.6]{LCHAO2025} proved that the anisotropic $p$-torsional rigidity satisfies a property analogous to the classical torsional rigidity. Let $u$ be a solution of (\ref{Eq:patr}), $1<p<\infty$ and $K, L \in \mathscr{K}^n$. Then the anisotropic $p$-torsional rigidity $\tau_{F,p}(K)$ is positively homogeneous of order $n+\frac{p}{p-1}$, that is, 
\begin{align}\label{EQ:Tpybb}
\tau_{F,p}(\lambda K)=\lambda^{n+\frac{p}{p-1}}\tau_{F,p}(K).
\end{align}
The anisotropic $p$-torsional rigidity $\tau_{F,p}(K)$ is translation invariant, namely, 
\begin{align}
\label{Eq:Tpyb}\tau_{F,p}(K+x)=\tau_{F,p}(K), \quad x \in \mathbb{R}^n.
\end{align}
The anisotropic $p$-torsional measure $S_{F,p}(K,\cdot)$ is also translation invariant. 

If $K, L \in \mathscr{K}^n$ and $K \subseteq L$, then 
$$\tau_{F,p}(K) \leq \tau_{F,p}(L).$$

Let $K,K_i \subset \mathbb{R}^n$ be bounded domains with $C^{2, \alpha}$ and let  $F$ be a norm in $\mathscr{I}_p$. If $K_i\to K$ in Hausdorff metric as $i \rightarrow \infty$. Then (see \cite[Lemma 3.11 and Lemma 3.12]{LCHAO2025})
\begin{align}\label{EQ:RSN}
S_{F,p}(K_i, \cdot) \rightarrow S_{F,p}(K, \cdot), \quad {\rm weakly \ on  \quad \mathbb{S}^{n-1}}.
\end{align}
and $\tau_{F,p}$ is continuous, i.e.,
\begin{align}\label{EQ:TDRSN}
\tau_{F,p}(K_i) \rightarrow \tau_{F,p}(K).
\end{align}

In addition, Pietra and Gavitone \cite[Remark 2.1]{PDG2014} used the anisotropic P\'{o}lya-Szeg\"{o} inequality to characterized the upper bound estimate of the anisotropic $p$-torsional rigidity.
\begin{lemma}[\cite{PDG2014}]\label{lemuppb} 
Let $1<p<+\infty$ and $K$ be a bounded open set of $\mathbb{R}^n$. Then 
\begin{align}\label{Eq:sjgj}
\tau_{F,p}(K) \leq\frac{p-1}{n(p-1)+p}n^{-\frac{1}{p-1}}\kappa_n^{-\frac{p}{n(p-1)}}V(K)^{\frac{n(p-1)+p}{n(p-1)}},
\end{align}
where $\kappa_n=|\mathcal{W}|$ and $|\mathcal{W}|$ denotes the Lebesgue measure of $\mathcal{W}=\left\{\xi \in \mathbb{R}^n: F^o(\xi)<1\right\}$ and $F^o(\xi) = \sup_{\eta \neq 0} \frac{\eta \cdot\xi}{F(\eta)}$.
\end{lemma}

Now, we give the variational formula of the anisotropic $p$-torsional rigidity with respect to the $L_q$-Minkowski combination.

\begin{lemma}\label{lem:LPbfgs}
Let $1<p<\infty$, $q \neq 0$ and $K$ be a convex body containing the origin in its interior, such that $\partial K$ is smooth  up to set of $(n-1)$-dimensional Hausdorff measure zero and $f: \mathbb{S}^{n-1} \rightarrow \mathbb{R}$ be a continuous function. For sufficiently small $\delta>0$ and each $t \in(-\delta, \delta)$, the continuous function $h_t: \mathbb{S}^{n-1} \rightarrow(0, \infty)$ is defined by
$$h_t(v)=(h_K(v)^q+t f(v)^q)^{\frac{1}{q}}, \ \ \ v \in \mathbb{S}^{n-1},$$
then 
$$
\lim _{t \rightarrow 0} \frac{\tau_{F,p}\left([h_t]\right)-\tau_{F,p}(K)}{t}=\frac{1}{q}\int_{\mathbb{S}^{n-1}} f^q(v) d S_{F,p,q}(K,v).
$$
\end{lemma}

\begin{proof}
Let $h_t(v) = \left(h_K(v)^q + t f(v)^q\right)^{\frac{1}{q}}$. By invoking equation (\ref{EQ:YXLG}) and applying the polar coordinate transformation formula, we derive
$$
\tau_{F,p}\left(\left[h_t\right]\right)=\int_K F^p(\nabla u(x)) d x=\int_{\mathbb{S}^{n-1}} \int_0^{\rho_{\left[h_t\right]}(v)} F^p(\nabla u(\xi v)) \xi^{n-1} d \xi d v=\int_{\mathbb{S}^{n-1}} \phi_t(v) d v,
$$
where $\phi_t(v)=\int_0^{\rho_{\left[h_t\right]}(v)}F^p(\nabla u(\xi v)) \xi^{n-1} d \xi$. According to the Lebesgue differentiation theorem, it follows that
$\left.\frac{\mathrm{d} \phi_t}{ d t}\right|_{t=0}=\phi_0^{\prime}$ exists. Let $\mathcal{I}=\left[\rho_K(v), \rho_{\left[h_t\right]}(v)\right]$, for all $v \in \mathbb{S}^{n-1}$, note that $F^p(\nabla u(\xi v))$ exists by the assumption on $F$, by Lemma \ref{lem:Lpbfyl}, we obtain
\begin{align*}
\lim _{t \rightarrow 0} \frac{\phi_t(v)-\phi_0(v)}{t} & =\lim _{t \rightarrow 0} \frac{1}{t} \int_\mathcal{I} F^p(\nabla u(\xi v)) \xi^{n-1} d \xi \\
& =\lim _{t \rightarrow 0} \frac{\rho_{\left[h_t\right]}(v)-\rho_K(v)}{t} \frac{1}{|\mathcal{I}|} \int_\mathcal{I} F^p(\nabla u(\xi v)) \xi^{n-1} d \xi \\
& =F^p(\nabla u(r_K(v))) \frac{f(\mathbf{g}_K(r_K(v)))^q}{q h_K(\mathbf{g}_K(r_K(v)))^q} \rho_K^n(v).
\end{align*}
Since for some $M>0$,
$$
 \frac{f(\mathbf{g}_K(r_K(v)))^q}{q h_K(\mathbf{g}_K(r_K(v)))^q} \rho_K(v)=\lim _{t \rightarrow 0} \frac{\rho_{\left[h_t\right]}(v)-\rho_K(v)}{t} \leq M,
$$ 
from Lemma \ref{lem:Lpbfyl}, we can conclude that the derivative is dominated by the integrable function 
$$M\left(\max _{v \in \mathbb{S}^{n-1}} \rho_K(v)\right)^{n-1} F^p(\nabla u(r_K(v))).$$
By applying the dominated convergence theorem to differentiate under the integral sign and combining Lemma \ref{lem:jxy1} with (\ref{Eq:gdfs}), we obtain
\begin{align*}
\int_{\mathbb{S}^{n-1}} \phi_0^{\prime}(v) d u & =\int_{\mathbb{S}^{n-1}} F^p(\nabla u(r_K(v))) \frac{f(\mathbf{g}_K(r_K(v)))^q}{q h_K(\mathbf{g}_K(r_K(v)))^q} \rho_K^n(v) d v \\
& =\frac{1}{q}\int_{\partial K} F^p(\nabla u(x)) f^q(\mathbf{g}_K(x))h_K^{1-q}(\mathbf{g}_K(r_K(v)))d\mathscr{H}^{n-1}(x)\\
& =\frac{1}{q}\int_{\mathbb{S}^{n-1}} f^q(v) d S_{F,p,q}(K,v).
\end{align*}
This concludes the proof of Lemma \ref{lem:LPbfgs}.
\end{proof}

From Lemma \ref{lem:LPbfgs}, if $1<p<\infty$, $q\neq 0$ and $K$ is a convex body in $\mathbb{R}^n$ that contains the origin in its interior, then the $L_q$ anisotropic $p$-torsional measure, $S_{F,p,q}(K,\cdot)$, of $K$ is a Borel measure on the unit sphere $\mathbb{S}^{n-1}$ defined for a Borel set $\eta\subset \mathbb{S}^{n-1}$ by
\begin{align}\label{Eq:lpcd}
S_{F,p,q}(K,\eta)=\int_{x\in\mathbf{g}_K^{-1}(\eta)}\langle x, \mathbf{g}_K(x)\rangle^{1-q} F^p\left(\nabla u(x)\right) d\mathscr{H}^{n-1}(x).
\end{align}

Obviously, the $L_q$ anisotropic $p$-torsional measure $S_{F,p,q}(K,\cdot)$ is absolutely continuous with respect to anisotropic $p$-torsional measure $S_{F,p}(K,\cdot)$ and has Radon-Nikodym derivative
\begin{align}\label{Eq:RDND}
\frac{dS_{F,p,q}(K,\cdot)}{dS_{F,p}(K,\cdot)}=h^{1-q}(K,\cdot).
\end{align}

From that $K_i\rightarrow K$ in the Hausdorff metric and that the convergence of support functions is uniform, the continuity of the anisotropic $p$-torsional measure (\ref{EQ:RSN}) together with (\ref{Eq:RDND}) implies the weak continuity of the $L_q$ anisotropic $p$-torsional measure, i.e.,
\begin{align}\label{EQ:LQRSN}
S_{F,p,q}(K_i, \cdot) \rightarrow S_{F,p,q}(K, \cdot), \quad {\rm weakly \ on  \quad \mathbb{S}^{n-1}}.
\end{align}

If $q=1$ in Lemma \ref{lem:LPbfgs}, we can obtain the variational formula of the anisotropic $p$-torsional rigidity with respect to the Minkowski combination.
\begin{corollary}\label{lembfgsn}
Let $1<p<\infty$ and $K$ be a convex body in $\mathbb{R}^n$ containing the origin in its interior, such that $\partial K$  is smooth up to set of $(n-1)$-dimensional Hausdorff measure zero and $f: \mathbb{S}^{n-1} \rightarrow \mathbb{R}$ be a continuous function. For sufficiently small $\delta>0$ and each $t \in(-\delta, \delta)$, the continuous function $h_t: \mathbb{S}^{n-1} \rightarrow(0, \infty)$ is defined by
$$h_t(v)=h_K(v)+t f(v), \ \ \ v \in \mathbb{S}^{n-1},$$
then 
$$
\lim _{t \rightarrow 0} \frac{\tau_{F,p}\left([h_t]\right)-\tau_{F,p}(K)}{t}=\int_{\mathbb{S}^{n-1}} f(v) d S_{F,p}(K,v).
$$
\end{corollary}

\section{A variational proof of the $L_q$-Minkowski problem for the anisotropic $p$-torsional measure}\label{SEC4}

Let $1<p<\infty$ and $q > 0$. For each non-zero finite Borel measure $\mu$ on $\mathbb{S}^{n-1}$, define the functional $\Psi_{F,p, q}: C^{+}(\mathbb{S}^{n-1}) \rightarrow \mathbb{R}^n$ by
$$
\Psi_{F,p, q}(f)=\frac{1}{q} \log \int_{\mathbb{S}^{n-1}} f(v)^q d\mu(v)-\frac{p-1}{n(p-1)+p} \log \tau_{F,p}([f]).
$$
A direct computation shows that $\Psi_{F,p, q}(f)$ is homogeneous of degree 0, i.e.,
$$
\Psi_{F,p, q}(t f)=\Psi_{F,p, q}(f), \quad \forall t>0, f \in C^{+}\left(\mathbb{S}^{n-1}\right).
$$
Note that for each $f \in C^{+}(\mathbb{S}^{n-1})$, from the definition of wulff shapes, we have $h_{[f]}\leq f$ and $[f]=[h_{[f]}]$, then
$$
\Psi_{F,p, q}(f) \geq \Psi_{F,p, q}(h_{[f]}).
$$

\begin{lemma}\label{Eq:lemazxzwt}
Let $1<p<\infty$, $1<q\neq n+\frac{p}{p-1}$ and $\mu$ be a nonzero finite Borel measure on $\mathbb{S}^{n-1}$ that is not concentrated in any closed hemisphere. If the minimization problem
\begin{align}\label{Eq:ZDXZ}
\inf \left\{\Psi_{F,p, q}(h):\quad \tau_{F,p}([h])=|\mu|,\quad h \in C^{+}(\mathbb{S}^{n-1})\right\},
\end{align}
has a solution $h_0 \in C^{+}(\mathbb{S}^{n-1})$, then there exists $K\in \mathscr{K}_o^n$ such that
\begin{align}
S_{F,p, q}\left(K, \cdot\right)=\mu.
\end{align}
\end{lemma}

\begin{proof}
Now, we may restrict our attention in the search of a minimizer to the set of all support functions. That is, $h_{K_0}$ is a minimizer to the optimization problem (\ref{Eq:ZDXZ}) if and only if $K_0$ is a minimizer to the following optimization problem:
\begin{align}\label{Eq:yhwrtu}
\inf \left\{\Psi_{F,p, q}(K_0):\quad \tau_{F,p}(K_0)=|\mu|,\quad K_0 \in \mathscr{K}_o^n\right\},
\end{align}
where $\Psi_{F,p,q}: \mathscr{K}_o^n \rightarrow \mathbb{R}$ is defined by
$$
\Psi_{F,p, q}(K_0)=\frac{1}{q} \log \int_{\mathbb{S}^{n-1}} h_{K_{0}}(v)^q d\mu(v)-\frac{p-1}{n(p-1)+p} \log \tau_{F,p}(K_0),
$$
for each $K_0 \in \mathscr{K}_o^n$. Clearly, the functional $\Psi_{F,p,q}$ is continuous with respect to the Hausdorff metric.

Suppose $K_0$ is a minimizer to (\ref{Eq:yhwrtu}), or equivalently $h_{K_0}$ is a minimizer to (\ref{Eq:ZDXZ}), i.e., $\tau_{F,p}(K_0)=|\mu|$ and
$$
\Psi_{F,p, q}(h_{K_0})=\inf \left\{\Psi_{F,p, q}(h): \tau_{F,p}([h])=|\mu|, h \in C^{+}(\mathbb{S}^{n-1})\right\},
$$
this yields that
$$
\Psi_{F,p, q}(h_{K_0}) \leq \Psi_{F,p, q}(h),
$$
for each $h \in C^{+}\left(\mathbb{S}^{n-1}\right)$. Let $f\in C^+(\mathbb{S}^{n-1})$ be an arbitrary continuous function. For $\delta>0$ small enough and $t \in(-\delta, \delta)$, define $h_t: \mathbb{S}^{n-1} \rightarrow \mathbb{R}$ by 
$$h_t(v)=(h_{K_0}(v)^q+t f(v)^q)^{\frac{1}{q}}.$$
Since $h_{K_0}$ is a minimizer of  (\ref{Eq:yhwrtu}) and by Lemma \ref{lem:LPbfgs}, we obtain
\begin{align*}
0 &= \frac{d}{dt} \bigg|_{t=0}\Psi_{F,p, q}(h_t)\\
&=\left.\frac{1}{q} \frac{ d }{ d t}\right|_{t=0}\left(\log \int_{\mathbb{S}^{n-1}} h_t(v)^q d \mu(v)\right)-\left.\frac{p-1}{n(p-1)+p}\frac{ d }{ d t}\right|_{t=0}\left(\log \tau_{F,p}([h_t])\right) \\
& =\frac{1}{q}\frac{\int_{\mathbb{S}^{n-1}} f^q(v) d \mu(v)}{\int_{\mathbb{S}^{n-1}} h^q_{K_0}d \mu(v)}-\frac{(p-1)}{q(n(p-1)+p) \tau_{F,p}([h_{K_0}])} \int_{\mathbb{S}^{n-1}} f^q(v) d S_{F,p,q}([h_{K_0}],v).
\end{align*}

Since $f\in C^+(\mathbb{S}^{n-1})$ is arbitrary and noting that $[h_{K_0}]=K_0$, we have 
$$
\frac{S_{F,p,q}(K_0, \cdot)}{|\mu|}=\frac{(n(p-1)+p)}{(p-1)\int_{\mathbb{S}^{n-1}} h_{K_0}^q d \mu} \mu(\cdot).
$$
Since measure $S_{F,p,q}(K_0, \cdot)$ is homogeneous of degree $n+\frac{p}{p-1}-q$ in $K_0$. Let 
$$K =\left(\frac{(p-1)\int_{\mathbb{S}^{n-1}} h_{K_0}^q d \mu(v)}{(n(p-1)+p)|\mu|}K_0\right)^{1/(n+\frac{p}{p-1}-q)},$$
then
$$
S_{F,p,q}\left(K, \cdot\right)=\mu.
$$
\end{proof}
 
\begin{theorem}\label{thm:elpmw} 
Let $1<p<\infty$, $1<q\neq n+\frac{p}{p-1}$ and $\mu$ be a nonzero finite Borel measure on $\mathbb{S}^{n-1}$ that is not concentrated on a great subsphere, then there exists a convex body $K \in \mathscr{K}_o^n$ such that
\begin{align*}
S_{F,p, q}(K, \cdot)=\mu.
\end{align*}
\end{theorem}

\begin{proof} 
Suppose $\{K_{0_i}\} \subset \mathscr{K}_o^n$ is a minimal sequence, i.e.,
\begin{align}\label{Eq:yhw66}
\inf \left\{\Psi_{F,p, q}(K_{0_i}):\quad \tau_{F,p}(K_{0_i})=|\mu|,\quad K_{0_i} \in \mathscr{K}_o^n\right\},
\end{align}
we claim that $K_{0_i}$ is uniformly bounded. If not, then there exists $u_i \in \mathbb{S}^{n-1}$ such that $\rho_{K_{0_i}}\left(u_i\right) \rightarrow \infty$ as $i \rightarrow \infty$. By the definition of support function, $\rho_{K_{0_i}}\left(u_i\right)\left(u_i \cdot v\right)_{+} \leq h_{K_{0_i}}(v)$, then
\begin{align*}
\Psi_{F,p, q}(K_{0_i})&=\frac{1}{q} \log \int_{\mathbb{S}^{n-1}} h_{K_{0_i}}(v)^q d\mu(v)-\frac{p-1}{n(p-1)+p} \log \tau_{F,p}(K_{0_i})\\
&\geq\frac{1}{q} \log\int_{\mathbb{S}^{n-1}} \rho_{K_{0_i}}\left(u_i\right)^q\left(u_i \cdot v\right)_{+}^q d \mu(v)-\frac{(p-1)\log(|\mu|)}{n(p-1)+p} \\
& =\frac{1}{q} \log\left(\rho_{K_{0_i}}(u_i)^q \int_{\mathbb{S}^{n-1}}\left(u_i \cdot v\right)_{+}^q d \mu(v)\right)-\frac{(p-1)\log(|\mu|)}{n(p-1)+p},
\end{align*}
where $(t)_{+}=\max \{t, 0\}$ for any $t \in \mathbb{R}$. Since $\mu$ is not concentrated on any closed hemisphere, there exists a positive constant $c_0$ such that
$$
\int_{\mathbb{S}^{n-1}}\left(u_i \cdot v\right)_{+}^q d \mu(v)>c_0.
$$
Therefore
$$
\Psi_{F,p, q}(K_{0_i})>\frac{1}{q} \log\left(\rho_{K_{0_i}}(u_i)^q c_0\right)-\frac{(p-1)\log(|\mu|)}{n(p-1)+p} \rightarrow \infty,
$$
as $i \rightarrow \infty$. But this contradicts (\ref{Eq:yhw66}). So we conclude that $K_{0_i}$ is uniformly bounded. By Blaschke selection theorem \cite[Theorem 1.8.7]{SRA2014}, there is a convergent subsequence of $K_{0_i}$, still denoted by $K_{0_i}$, converges to a compact convex set $K$ of $\mathbb{R}^n$. By the continuity of the anisotropic $p$-torsional rigidity, we get
$$
|\mu|=\lim _{i \rightarrow \infty} \tau_{F,p}\left(K_{0_i}\right)= \tau_{F,p}(K_0).
$$
From (\ref{Eq:sjgj}), we have 
\begin{align*}
|\mu|=\tau_{F,p}(K_0) \leq\frac{p-1}{n(p-1)+p}n^{-\frac{1}{p-1}}\kappa_n^{-\frac{p}{n(p-1)}}V(K_0)^{\frac{n(p-1)+p}{n(p-1)}}.
\end{align*}
Consequently, we have,
\begin{align*} 
V(K_0)\geq\left(|\mu|\frac{n(p-1)+p}{p-1}n^{\frac{1}{p-1}}\kappa_n^{\frac{p}{n(p-1)}}\right)^{\frac{n(p-1)}{n(p-1)+p}}>0.
\end{align*}
Thus $K_0$ is a convex body in $\mathbb{R}^n$. 

Therefore $K_0$ is a minimizer for the minimum problem (\ref{Eq:yhwrtu}), Because of the minimum problem (\ref{Eq:yhwrtu}) and (\ref{Eq:ZDXZ}) are equivalent. So $h_{K_0}$ is the solution to problem (\ref{Eq:ZDXZ}). Combining with Lemma \ref{Eq:lemazxzwt}, we can get the desired result.
\end{proof}

\section{The $L_q$-Minkowski problem of the anisotropic $p$-torsional rigidity for the general measure with $0<q<1$} 

Suppose that $\mu$ is a finite discrete measure on $\mathbb{S}^{n-1}$ which is not concentrated on any closed hemisphere of $\mathbb{S}^{n-1}$, that is 
$$
\mu=\sum_{k=1}^N \alpha_k \delta_{u_k},
$$
 where $\delta_{u_k}$ is Kronecker delta, $\alpha_1, \ldots, \alpha_N>0$ and $u_1, \ldots, u_N \in \mathbb{S}^{n-1}$ are not concentrated on any closed hemisphere. For $0<q<1$, we consider the following minimizing problem
\begin{align}\label{Eq:zxzwt}
\inf \left\{\sup _{\xi \in[f]_\mu} \Phi_{f, \mu}(\xi): f \in C^{+}\left(\mathbb{S}^{n-1}\right) \text { and } \tau_{F,p}([f]_\mu)=1\right\},
\end{align}
where $[f]_\mu$ denotes the Aleksandrov body associated to $(f, \operatorname{supp}(\mu))$ and $\Phi_{f, \mu}:[f]_\mu \rightarrow \mathbb{R}$ is defined by
\begin{align}\label{Eq:hsnhsu}
\Phi_{f, \mu}(\xi)=\int_{\mathbb{S}^{n-1}}(f(u)-\xi \cdot u)^q d \mu(u)=\int_{\operatorname{supp}(\mu)}(f(u)-\xi \cdot u)^q d \mu(u).
\end{align} 

In order to prove that this minimum problem is exactly the solution of the anisotropic $p$-torsional rigidity for the $L_q$-Minkowski problem, the following lemma is necessary. This lemma was proved by Jian and Lu \cite[Lemma 3.3]{JL2019} for the first time.

\begin{lemma}\label{lem:yjma1}\cite[Lemma 3.3]{JL2019}
If $0<q<1$. Let $\mu$ be a finite discrete Borel measure on $\mathbb{S}^{n-1}$ which is not concentrated in any closed hemisphere of $\mathbb{S}^{n-1}$. Then $\Phi_{f, \mu}(\xi)$ is strictly concave for every non-negative continuous function $f$ on $\mathbb{S}^{n-1}$ with $[f]_\mu$ has a nonempty interior, there is an unique point $\xi_{f}\in[f]_\mu$, such that
\begin{equation}
\Phi_{f, \mu}(\xi_{f})=\sup_{\xi \in [f]_\mu}\Phi_{f, \mu}(\xi),
\end{equation}
where $\xi_f$ depends continuously on $f$.
\end{lemma}
\begin{lemma}\label{lem:52a}
 Let $1<p<\infty$, $0<q<1$ and $\mu$ be a finite discrete measure on $\mathbb{S}^{n-1}$ which is not concentrated in any closed hemisphere of $\mathbb{S}^{n-1}$, then there exists a function $h \in C^{+}\left(\mathbb{S}^{n-1}\right)$ with $\xi_h=o$ and $\tau_{F,p}\left([h]_\mu\right)=1$ such that
$$
\Phi_{h, \mu}(o)=\inf \left\{\sup_{\xi \in[f]_\mu} \Phi_{f, \mu}(\xi): f \in C^{+}\left(\mathbb{S}^{n-1}\right) \text { and } \tau_{F,p}([f]_\mu)=1\right\}.
$$
\end{lemma}
\begin{proof}
Let $\left\{f_k\right\} \subset C^{+}\left(\mathbb{S}^{n-1}\right)$, $\tau_{F,p}\left([f_k]_\mu\right)=1$ and
\begin{align}\label{Eq:jxazsdp}
\lim _{k \rightarrow+\infty} \Phi_{f_k, \mu}\left(\xi_{f_k}\right)=\inf \left\{\sup _{\xi \in[f]_\mu} \Phi_{f, \mu}(\xi): g \in C^{+}\left(\mathbb{S}^{n-1}\right) \text { and } \tau_{F,p}([f]_\mu)=1\right\}.
\end{align}
For simplicity, we write $h_k=h_{[f_k]_\mu}$, thus 
$$
h_k(u) \leq f_k(u)
$$
for $u \in \operatorname{supp}(\mu)$, then $[h_k]_\mu=\left[h_{[f_k]_\mu}\right]_\mu=[f_k]_\mu$. Since $o \in \operatorname{int}[f_k]_\mu$, we have $h_k \in C^{+}\left(\mathbb{S}^{n-1}\right)$. Thus for any $\xi \in[h_k]_\mu=[f_k]_\mu$, by (\ref{Eq:hsnhsu}), we get
\begin{align}\label{Eq:Wswbj}
\nonumber \Phi_{h_k, \mu}(\xi) & =\int_{\operatorname{supp}(\mu)}\left(h_k(u)-\xi \cdot u\right)^q d \mu(u) \\
\nonumber& \leq \int_{\operatorname{supp}(\mu)}\left(f_k(u)-\xi \cdot u\right)^q d \mu(u) \\
& =\Phi_{f_k, \mu}(\xi).
\end{align}
Thus by (\ref{Eq:Wswbj}), we obtain
$$
\sup _{\xi \in[h_k]_\mu} \Phi_{h_k, \mu}(\xi) \leq \sup _{\xi \in[f_k]_\mu} \Phi_{f_k, \mu}(\xi).
$$
Consequently,
\begin{align}\label{Eq:jxsda}
\lim _{k \rightarrow+\infty} \sup _{\xi \in[h_k]_\mu} \Phi_{h_k, \mu}(\xi) \leq \lim _{k \rightarrow+\infty} \sup _{\xi \in[f_k]_\mu} \Phi_{f_k, \mu}(\xi).
\end{align}
Due to $h_k \in C^{+}\left(\mathbb{S}^{n-1}\right)$ and $\tau_{F,p}\left([h_k]_\mu\right)=\tau_{F,p}\left([f_k]_\mu\right)=1$, by (\ref{Eq:jxazsdp}), we get
$$
\lim _{k \rightarrow+\infty} \sup _{\xi \in[h_k]_\mu} \Phi_{h_k, \mu}(\xi) \geq \lim _{k \rightarrow+\infty} \sup _{\xi \in[f_k]_\mu} \Phi_{f_k, \mu}(\xi),$$
from this and (\ref{Eq:jxsda}), we deduce
\begin{align}\label{Eq:dhjxs}
\lim _{k \rightarrow+\infty} \sup _{\xi \in[h_k]_\mu} \Phi_{h_k, \mu}(\xi)=\lim _{k \rightarrow+\infty} \sup _{\xi \in[f_k]_\mu} \Phi_{f_k, \mu}(\xi).
\end{align}
From Lemma \ref{lem:yjma1}, (\ref{Eq:dhjxs}) and (\ref{Eq:jxazsdp}), it follows that
\begin{align*}
\lim _{k \rightarrow+\infty}\Phi_{h_k, \mu}(\xi_{h_k})=&\lim _{k \rightarrow+\infty}\sup_{\xi \in [h_k]_\mu}\Phi_{h_k, \mu}(\xi)\\
=&\lim _{k \rightarrow+\infty} \sup _{\xi \in[f_k]_\mu} \Phi_{f_k, \mu}(\xi)\\
=&\inf \left\{\sup _{\xi \in[f]_\mu} \Phi_{f, \mu}(\xi): g \in C^{+}\left(\mathbb{S}^{n-1}\right) \text { and } \tau_{F,p}([f]_\mu)=1\right\}.
\end{align*}
In view of $[h_k]_\mu=[f_k]_\mu$, we obtain $h_k=h_{[f_k]_\mu}=h_{[h_k]_\mu}$, where $h_k$ is the support function of $[h_k]_\mu$. Since $\tau_{F,p}$ is translation invariant, then for $x \in \mathbb{R}^n$,  $\tau_{F,p}([h_k]_\mu+x)=\tau_{F,p}([h_k]_\mu)=1$. By (\ref{Eq:hsnhsu}), we derive for $x \in \mathbb{R}^n$,
\begin{align*}
\Phi_{h_{([h_k]_\mu+x)}, \mu}(\xi_{h_k}+x)
=\int_{\operatorname{supp}(\mu)}\left(h_{([h_k]_\mu+x)}(u)-(\xi_{h_k}+x) \cdot u\right)^q d \mu(u)
=\Phi_{h_k, \mu}(\xi_{h_k}).
\end{align*}

Therefore we can select a sequence, which will also be represented by$\left\{h_k\right\} \subset C^{+}\left(\mathbb{S}^{n-1}\right)$, with $\tau_{F,p}\left([h_k]_\mu\right)=1$ and $\xi_{h_k}=o$, such that
$$
\lim _{k \rightarrow+\infty} \Phi_{h_k, \mu}(o)=\inf \left\{\sup _{\xi \in[f]_\mu} \Phi_{f, \mu}(\xi): g \in C^{+}\left(\mathbb{S}^{n-1}\right) \text { and } \tau_{F,p}([f]_\mu)=1\right\}.
$$
We assert that the sequence $\left\{h_k\right\}$ is uniformly bounded on $\mathbb{S}^{n-1}$. Should this assertion not hold, it would be possible to extract a subsequence of $\left\{h_k\right\}$, which we shall also denote as $\left\{h_k\right\}$, such that
$$
\lim _{k \rightarrow+\infty} \sup _{u \in \mathbb{S}^{n-1}} h_k(u)=+\infty.
$$
Let $R_k=\sup _{u \in \mathbb{S}^{n-1}} h_k(u)=h_k\left(u_k\right)$ for some $u_k \in \mathbb{S}^{n-1}$. Given that $\left\{u_k\right\} \subset \mathbb{S}^{n-1}$,  it follows from the compactness of $\mathbb{S}^{n-1}$ that there exists a convergent subsequence, denoted by $\left\{u_k\right\}$, under the assumption that
$$
\lim _{k \rightarrow+\infty} u_k=u_0 \in \mathbb{S}^{n-1}.
$$
Since $\operatorname{supp}(\mu)$ is not concentrated on any closed hemisphere, there exist some $\zeta \in \operatorname{supp}(\mu)$ such that
$$
\zeta \cdot u_0>0.
$$
Define  $\gamma=\frac{1}{2}\left(\zeta \cdot u_0\right)>0$. It follows that there exists $k_0 \in \mathbb{N}$ such that for all $k \geq k_0$,
$$
\zeta \cdot u_k>\gamma.
$$
Note that $R_k u_k \in[h_k]_\mu$. As a result, for all $k \geq k_0$, it follows that
$$
h_k\left(\zeta\right) \geq R_k\left(\zeta \cdot u_k\right)>R_k \gamma.
$$
Given that $\mu$ is a finite discrete measure, it follows that for all $k \geq k_0$ and $0<q<1$, we obtain
\begin{align}\label{Eq:mdfaz}
\nonumber\lim _{k \rightarrow+\infty} \Phi_{h_k, \mu}(o) & =\lim _{k \rightarrow+\infty} \int_{\mathbb{S}^{n-1}} h_k^q(u) d \mu(u) \\
\nonumber& \geq \lim _{k \rightarrow+\infty} h_k^q\left(\zeta\right) \mu\left(\zeta\right) \\
& >\lim _{k \rightarrow+\infty}\left(R_k \gamma\right)^q \mu\left(\zeta\right)=+\infty.
\end{align}
Let $h^{\prime} \in C^{+}\left(\mathbb{S}^{n-1}\right)$ and $\tau_{F,p}([h^{\prime}]_\mu)=1$. Then 
$$
\lim _{k \rightarrow+\infty} \Phi_{h_k, \mu}(o) \leq \Phi_{h^{\prime}, \mu}\left(\xi_{h^{\prime}}\right)=\int_{\mathbb{S}^{n-1}}\left(h^{\prime}(u)-\xi_{h^{\prime}} \cdot u\right)^q d \mu(u)<+\infty,
$$
which contradicts to (\ref{Eq:mdfaz}). Thus $\left\{h_k\right\}$ is uniformly bounded.

 By the Blaschke Selection Theorem \cite[Theorem 1.8.7]{SRA2014}, then $\left\{h_k\right\}$ has a convergent subsequence, likewise denoted by $\left\{h_k\right\}$, letting $h_k \rightarrow h$ on $\mathbb{S}^{n-1}$ as $k \rightarrow+\infty$. This yields $h \geq 0,[h_k]_\mu \rightarrow[h]_\mu$ and $h=h_{[h]_\mu}$. The continuity of $\tau_{F,p}$ assures that $\tau_{F,p}([h]_\mu)=1$. 
 From (\ref{Eq:sjgj}), we have 
\begin{align*}
1=\tau_{F,p}([h]_\mu) \leq\frac{p-1}{n(p-1)+p}n^{-\frac{1}{p-1}}\kappa_n^{-\frac{p}{n(p-1)}}V([h]_\mu)^{\frac{n(p-1)+p}{n(p-1)}}.
\end{align*}
Consequently, we have
\begin{align*} 
V([h]_\mu)\geq\left(\frac{n(p-1)+p}{p-1}n^{\frac{1}{p-1}}\kappa_n^{\frac{p}{n(p-1)}}\right)^{\frac{n(p-1)}{n(p-1)+p}}>0.
\end{align*}
This indicates that $[h]_\mu$ is a convex body in $\mathbb{R}^n$ and $h=h_{[h]_\mu}$. Furthermore, it directly follows from Lemma \ref{lem:yjma1} that
$$
o=\lim _{k \rightarrow+\infty} \xi_{h_k}=\xi_h \in \operatorname{int}[h]_\mu.
$$
Namely, $h>0$. Consequently, based on (\ref{Eq:jxazsdp}), we obtain
$$
\Phi_{h, \mu}(o)=\inf \left\{\sup _{\xi \in[f]_\mu} \Phi_{f, \mu}(\xi): g \in C^{+}\left(\mathbb{S}^{n-1}\right) \text { and } \tau_{F,p}([f]_\mu)=1\right\}.
$$
This concludes the proof.
\end{proof}

\begin{lemma}\label{lem:k98ct}
 Let $1<p<\infty$, $0<q<1$ and $\mu$ be a finite discrete measure on $\mathbb{S}^{n-1}$ which is not concentrated in any closed hemisphere of $\mathbb{S}^{n-1}$, then there exists a function $h \in C^{+}\left(\mathbb{S}^{n-1}\right)$ and a constant $c>0$ such that
$$
d\mu=\lambda d S_{F,p,q}\left([h]_\mu, \cdot\right),
$$
where
$$
\lambda=\frac{p-1}{n(p-1)+p} \int_{\mathbb{S}^{n-1}} h^q(u) d \mu(u).
$$
\end{lemma}
\begin{proof}
By Lemma \ref{lem:yjma1}, we infer the existence of a function $h \in C^{+}\left(\mathbb{S}^{n-1}\right)$ with $\xi_h=o$ and $\tau_{F,p}\left([h]_\mu\right)=1$ such that
$$
\Phi_{h, \mu}(o)=\inf \left\{\sup _{\xi \in[f]_\mu} \Phi_{f, \mu}(\xi): f \in C^{+}\left(\mathbb{S}^{n-1}\right) \text { and } \tau_{F,p}([f]_\mu)=1\right\}.
$$
Let $\varepsilon>0$ be sufficiently small, for any $f \in C^+(\mathbb{S}^{n-1})$ and $t \in(-\varepsilon, \varepsilon)$, define
$$
\vartheta_t=h+t f,
$$
ensuring that $\vartheta_t \in C^{+}\left(\mathbb{S}^{n-1}\right)$. From Corollary \ref{lembfgsn}, it follows that
\begin{align}\label{Eq:bfgsq}
\nonumber\lim _{t \rightarrow 0} \frac{\tau_{F,p}([\vartheta_t]_\mu)-\tau_{F,p}([h]_\mu)}{t}
= &\int_{\mathbb{S}^{n-1}} f(u) d S_{F,p}([h]_\mu, u)\\ =&\int_{\operatorname{supp}(\mu)} f(u) d S_{F,p}([h]_\mu, u).
\end{align}

Let $\psi_t=\varphi (t) \vartheta_t$, where
$$
 \varphi(t)=\tau_{F,p}([\vartheta_t]_\mu)^{-{\frac{p-1}{n(p-1)+p}}}.
$$
Obviously, $\psi_t \in C^{+}(\mathbb{S}^{n-1})$, $\tau_{F,p}([\psi_t]_\mu)=1$ and $\psi_0=h$. By (\ref{Eq:bfgsq}), we deduce
\begin{align}\label{Eqtcra}
\nonumber\lim _{t \rightarrow 0} \frac{\psi_t-\psi_0}{t}=&-h(u)\frac{p-1}{n(p-1)+p}\tau_{F,p}([h]_\mu)^{-{\frac{p-1}{n(p-1)+p}}-1}\int_{\mathbb{S}^{n-1}} f(u) d S_{F,p}([h]_\mu, u)+f\\
=&-h(u)\frac{p-1}{n(p-1)+p}\int_{\mathbb{S}^{n-1}} f(u) d S_{F,p}([h]_\mu, u)+f.
\end{align}
Let $\xi(t)=\xi_{\psi_t}$ and
\begin{align}\label{Eq:pduitr5}
\nonumber\Phi_\mu(t) & =\sup _{\xi \in\left[\psi_t\right]_\mu} \int_{\mathbb{S}^{n-1}}\left(\psi_t(u)-\xi \cdot u\right)^q d \mu(u) \\
& =\int_{\mathbb{S}^{n-1}}\left(\psi_t(u)-\xi(t) \cdot u\right)^q d \mu(u).
\end{align}
Since $\xi(t) \in \operatorname{int}\left[\psi_t\right]_\mu$,  by (\ref{Eq:pduitr5}), we obtain
\begin{align}\label{Eq:p09uy}
\int_{\mathbb{S}^{n-1}}\left(\psi_t(u)-\xi(t) \cdot u\right)^{q-1} u_k d \mu(u)=0,
\end{align}
for $k=1, \ldots, n$, where $u=(u_1, \ldots, u_n)^T$. Noting $\xi(0)=\xi_h=o$ and taking $t=0$ in (\ref{Eq:p09uy}), we get
\begin{align}\label{Eq:9idq}
\int_{\mathbb{S}^{n-1}} h^{q-1}(u) u_k d \mu(u)=0,
\end{align}
for $k=1, \ldots, n$. Hence
\begin{align}\label{Eq:pcdw}
\int_{\mathbb{S}^{n-1}} h^{q-1}(u) u d \mu(u)=0.
\end{align}
Let
$$
H_k\left(t, \xi_1, \ldots, \xi_n\right)=\int_{\mathbb{S}^{n-1}}\left(\psi_t(u)-\left(\xi_1 u_1+\cdots+\xi_n u_n\right)\right)^{q-1} u_k d \mu(u),
$$
for $k=1, \ldots, n$. Then 
$$
\frac{\partial H_k}{\partial \xi_l}=(1-q) \int_{\mathbb{S}^{n-1}}\left(\psi_t(u)-\left(\xi_1 u_1+\cdots+\xi_n u_n\right)\right)^{q-2} u_k u_l d \mu(u).
$$
Let $H=\left(H_1, \ldots, H_n\right)$ and $\xi=\left(\xi_1, \ldots, \xi_n\right)$. Thus
$$
\left(\left.\frac{\partial H}{\partial \xi}\right|_{(0, \ldots, 0)}\right)_{n \times n}=(1-q) \int_{\mathbb{S}^{n-1}} h^{q-2}(u) u u^T d \mu(u),
$$
where $u u^T$ is a $n \times n$ matrix.

Since $\mu$ is not concentrated on any closed hemisphere, $\operatorname{supp}(\mu)$ spans the whole space $\mathbb{R}^n$. Consequently, for any $x \in \mathbb{R}^n$ with $x \neq 0$, there exists a $u_{i_0} \in \operatorname{supp}(\mu)$ such that $u_{k_0} \cdot x \neq 0$. Therefore for $0<q<1$, we derive that
$$
\begin{aligned}
x^T\left(\left.\frac{\partial H}{\partial \xi}\right|_{(0, \ldots, 0)}\right) x 
= & x^T\left((1-q) \int_{\mathbb{S}^{n-1}} h^{q-2}(u) u u^T d \mu(u)\right) x \\
= & (1-q) \int_{\mathbb{S}^{n-1}} h^{q-2}(u)(x \cdot u)^2 d \mu(u) \\
\geq & (1-q) h^{q-2}\left(u_{k_0}\right)\left(x \cdot u_{k_0}\right)^2 \mu\left(u_{k_0}\right)>0.
\end{aligned}
$$
This indicates that $\left(\left.\frac{\partial H}{\partial \xi}\right|_{(0, \ldots, 0)}\right)$ is positive definite, which consequently implies that
$$
\operatorname{det}\left(\left.\frac{\partial H}{\partial \xi}\right|_{(0, \ldots, 0)}\right) \neq 0.
$$
From this, the facts that for $k=1, \ldots, n, H_k(0, \ldots, 0)=0$ follows by the equation (\ref{Eq:9idq}) and $\frac{\partial H_k}{\partial \xi_l}$ is continuous on a neighborhood of $(0, \ldots, 0)$ for all $1 \leq k, l \leq n$ and the implicit function theorem, we can conclude that
$$
\xi^{\prime}(0)=\left(\xi_1^{\prime}(0), \ldots, \xi_n^{\prime}(0)\right)
$$
exists.

Note that $\Phi_\mu(0)=\Phi_{h, \mu}(o)$ and $\Phi_\mu(t)=\Phi_{\psi_t, \mu}\left(\xi_{\psi_t}\right)$ with $\psi_t \in C^{+}\left(\mathbb{S}^{n-1}\right)$ and $\tau_{F,p}([\psi_t]_\mu)=1$. By Lemma \ref{lem:52a}, it follows that
$$
\Phi_\mu(t) \geq \Phi_\mu(0),
$$
i.e., $\Phi_\mu(0)$ is an extreme value of $\Phi_\mu(t)$. Therefore from (\ref{Eqtcra}) and (\ref{Eq:pcdw}), we obtain
\begin{align*}
0= & \frac{1}{q} \frac{d}{dt} \bigg|_{t=0}\Phi_\mu(t) \\
= & \int_{\mathbb{S}^{n-1}} h^{q-1}(u)\left(-h(u)\frac{p-1}{n(p-1)+p}\int_{\mathbb{S}^{n-1}} f(u) d S_{F,p}([h]_\mu, u)+f(u)-\xi^{\prime}(0) \cdot u\right) d \mu(u) \\
= & -\frac{p-1}{n(p-1)+p} \int_{\mathbb{S}^{n-1}} h^q(u) d \mu(u) \int_{\mathbb{S}^{n-1}} f(u) d S_{F,p}\left([h]_\mu, u\right)+\int_{\mathbb{S}^{n-1}} h^{q-1}(u) f(u) d \mu(u) \\
& -\int_{\mathbb{S}^{n-1}} \xi^{\prime}(0) \cdot h^{q-1}(u) u d \mu(u) \\
= & \int_{\mathbb{S}^{n-1}} h^{q-1}(u) f(u) d \mu(u)-\frac{p-1}{n(p-1)+p} \int_{\mathbb{S}^{n-1}} h^q(u) d \mu(u) \int_{\mathbb{S}^{n-1}} f(u) d S_{F,p}([h]_\mu, u).
\end{align*}
Namely, for all $f \in C^+\left(\mathbb{S}^{n-1}\right)$, 
$$
\int_{\mathbb{S}^{n-1}} h^{q-1}(u) f(u) d \mu(u)=\frac{p-1}{n(p-1)+p} \int_{\mathbb{S}^{n-1}} h^q(u) d \mu(u) \int_{\mathbb{S}^{n-1}} f(u) d S_{F,p}([h]_\mu, u).
$$
Since $h=h_{[h]_\mu}$, then
$$
d \mu(u)=\lambda d S_{F,p,q}\left([h]_\mu, u\right),
$$
where 
$$
\lambda=\frac{p-1}{n(p-1)+p} \int_{\mathbb{S}^{n-1}} h^q(u) d \mu(u).
$$
\end{proof}

\begin{theorem}
Let $1<p<\infty$, $0<q<1$ and $\mu$ be a finite Borel measure on $\mathbb{S}^{n-1}$ which is not concentrated in any closed hemisphere. Then there exists a convex body $K$ in $\mathbb{R}^n$, such that
\begin{equation*}
d\mu=\lambda dS_{F,p,q}(K, \cdot),
\end{equation*}
where $\lambda$ is a positive constant depending on $p$, $q$, $n$ and $\mu$.
\end{theorem}

\begin{proof} 
In accordance with the proof detailed in \cite[Theorem 8.2.2]{SRA2014}, given a finite Borel measure $\mu$ on $\mathbb{S}^{n-1}$ that is not concentrated within any closed hemisphere, it is feasible to construct a sequence of finite discrete measures $\{\mu_k\}$ on $\mathbb{S}^{n-1}$ such that $\mu_k(\mathbb{S}^{n-1}) = \mu(\mathbb{S}^{n-1})$ and $\mu_k$ converges weakly to $\mu$ as $k$ approaches infinity. Specifically, for sufficiently large values of $k$, $\mu_k$ will also not be concentrated in any closed hemisphere. According to Lemma \ref{lem:k98ct}, for each $\mu_k$, there exists a function $h_k \in C^{+}(\mathbb{S}^{n-1})$ and a positive constant $\lambda_k$, satisfying the conditions that  
\begin{align}\label{eq:lxsscg}
\mu_k=\lambda_k S_{F,p,q}([h_k]_{\mu_k}, \cdot),
\end{align}
where
\begin{align}\label{eq:lx66g}
\lambda_k=\frac{p-1}{n(p-1)+p}\int_{\mathbb{S}^{n-1}} h_k^q(u) d \mu_k(u).
\end{align}
Moreover, $h_k$ satisfies that $\xi_{h_k}=o, \tau_{F,p}([h_k]_{\mu_k})=1$ and
$$
\Phi_{h_k, \mu_k}(o)=\inf \left\{\sup _{\xi \in[f]_{\mu_k}} \Phi_{f, \mu_k}(\xi): f \in C^{+}\left(\mathbb{S}^{n-1}\right) \text { and } \tau_{F,p}([f]_{\mu_k})=1\right\},
$$
where
$$
[f]_{\mu_k}=\bigcap_{u \in \operatorname{supp}\left(\mu_k\right)}\left\{\xi \in \mathbb{R}^n: \xi \cdot u \leq f(u)\right\},
$$
and
$$
\Phi_{f, \mu_k}(\xi)=\int_{\mathbb{S}^{n-1}}(f(u)-\xi \cdot u)^q d \mu_k(u)=\int_{\operatorname{supp}\left(\mu_k\right)}(f(u)-\xi \cdot u)^q d \mu_k(u).
$$

The subsequent proof closely mirrors the approach outlined in \cite{JL2019,FZH2020}. Nonetheless, in the interest of ensuring a comprehensive and rigorous demonstration, we provide a detailed proof process. 

Define $g_k = \Phi_{h_k, \mu_k}(o)$. We aim to demonstrate that $g_k$ is uniformly bounded. Regarding the Aleksandrov body corresponding to $\left(1, \operatorname{supp}\left(\mu_k\right)\right)$, we denote it by $[1]_{\mu_k}$. Let  
$$
\bar{f}_k=(\tau_{F,p}([1]_{\mu_k}))^{-\frac{p-1}{n(p-1)+p}}.
$$
We observe that $\left[\bar{f}_k\right]_{\mu_k} = \bar{f}_k[1]_{\mu_k}$. It follows directly that $\tau_{F,p}([\bar{f}_k]_{\mu_k})=1$. Given that $\mu_k\left(\mathbb{S}^{n-1}\right) = \mu\left(\mathbb{S}^{n-1}\right)$, we get
\begin{align}\label{Eq:gdeyzj}
\nonumber g_k & =\Phi_{h_k, \mu_k}(o) \\
\nonumber & \leq \sup _{\xi \in\left[\bar{f}_k\right]_{\mu_k}} \int_{\mathbb{S}^{n-1}}\left(\bar{f}_k(u)-\xi \cdot u\right)^q d \mu_k(u) \\
\nonumber & \leq \int_{\mathbb{S}^{n-1}} {\rm diam}([\bar{f}_k]_{\mu_k})^q d \mu_k(u) \\
\nonumber & ={\rm diam}\left(\left[\bar{f}_k\right]_{\mu_k}\right)^q \mu\left(\mathbb{S}^{n-1}\right) \\
& =\bar{f}_k^q {\rm diam}\left([1]_{\mu_k}\right)^q \mu\left(\mathbb{S}^{n-1}\right).
\end{align}
To demonstrate that $g_k$ is uniformly bounded, we will establish that the diameter of $[1]_{\mu_k}$, denoted as ${\rm diam}([1]_{\mu_k})$, is uniformly bounded. If this is not the case, it would be possible to identify a sequence $\left\{\xi_k\right\}$ such that $\xi_k \in[1]_{\mu_k}$ and
$$
\lim _{k \rightarrow+\infty}\left|\xi_k\right|=+\infty.
$$
Let $\bar{\xi}_k=\frac{\xi_k}{\left|\xi_k\right|} \in \mathbb{S}^{n-1}$. By the compactness of $\mathbb{S}^{n-1}$, we may assume, without loss of generality, that
$$
\lim _{k \rightarrow+\infty} \bar{\xi}_k=\bar{\xi} \in \mathbb{S}^{n-1}.
$$
Given that $\operatorname{supp}(\mu)$ is not concentrated in any closed hemisphere, there exists $\bar{x} \in \operatorname{supp}(\mu)$, such that
\begin{align}\label{Eq:yxbm}
\bar{\xi} \cdot \bar{x}>0.
\end{align}
Let $O(\bar{x}) \subset \mathbb{S}^{n-1}$ denote an arbitrary neighborhood of $\bar{x}$. Consequently, it follows that 
$$
\liminf _{k \rightarrow+\infty} \mu_k(O(\bar{x})) \geq \mu(O(\bar{x}))>0.
$$
Observe that for sufficiently large values of $k$,
$$
O(\bar{x}) \bigcap \operatorname{supp}\left(\mu_k\right) \neq \emptyset, \text{ for infinitely many } k,
$$
which allows us to construct a sequence $\left\{\bar{x}_{k_l}\right\}$ such that
$$
\bar{x}_{k_l} \in \operatorname{supp}\left(\mu_{k_l}\right) \text { and } \lim _{i \rightarrow+\infty} \bar{x}_{k_l}=\bar{x}.
$$
It is important to note that $\xi_{k_l} \in[1]_{\mu_{k_l}}$. Consequently,
$$
\xi_{k_l} \cdot \bar{x}_{k_l} \leq h_{[1]_{\mu_{k_l}}}\left(\bar{x}_{k_l}\right) \leq 1,
$$
which implies that,
$$
\bar{\xi}_{k_l} \cdot \bar{x}_{k_l} \leq \frac{1}{\left|\xi_{k_l}\right|}.
$$
By taking the limit, we obtain
$$
\bar{\xi} \cdot \bar{x} \leq 0,
$$
which contradicts (\ref{Eq:yxbm}). Thus there exists a positive constant  $C>0$, such that for all $k \in \mathbb{N}$,
\begin{align}\label{Eq:idajs}
{\rm diam}\left([1]_{\mu_k}\right) \leq C.
\end{align}
Given that $\mathbb{B}^n \subset[1]_{\mu_k}$ for each $k \in \mathbb{N}$, it follows that for $1<p<\infty$, we have
\begin{align}\label{Eq:qidajs}
\bar{f}_k \leq(\tau_{F,p}(\mathbb{B}^n))^{-\frac{p-1}{n(p-1)+p}}.
\end{align}
From (\ref{Eq:gdeyzj}), (\ref{Eq:idajs}) and (\ref{Eq:qidajs}), it can be deduced that for $0<q<1$ and all $k \in \mathbb{N}$, 
\begin{align}\label{Eq:q66js}
g_k \leq C^q \tau_{F,p}(\mathbb{B}^n))^{-\frac{q(p-1)}{n(p-1)+p}} \mu\left(\mathbb{S}^{n-1}\right).
\end{align}
Thus the sequence $g_k$ is uniformly bounded.

Now, we shall prove that the sequence $\left\{h_k\right\}$ is uniformly bounded on $\mathbb{S}^{n-1}$. Otherwise, there exists a subsequence $\left\{h_{k_l}\right\} \subset \left\{h_k\right\}$ such that
$$
\lim _{l \rightarrow+\infty} \sup _{u \in \mathbb{S}^{n-1}} h_{k_l}(u)=+\infty.
$$
Let $R_{k_l}=\sup _{u \in \mathbb{S}^{n-1}} h_{k_l}(u)=h_{k_l}\left(u_{k_l}\right)$. Since $\left\{u_{k_l}\right\} \subset \mathbb{S}^{n-1}$, from the compactness of $\mathbb{S}^{n-1}$, without loss of generality we can assume
$$
\lim _{l \rightarrow+\infty} u_{k_l}=u_0 \in \mathbb{S}^{n-1}.
$$
Since $\operatorname{supp}(\mu)$ is not concentrated in any closed hemisphere, there exists $v_0 \in \operatorname{supp}(\mu)$ such that
$$
v_0 \cdot u_0>0.
$$
Let $O\left(v_0\right)$ be a small neighborhood of $v_0$ such that for all $u \in O\left(v_0\right)$, we have that
$$
u \cdot u_0>0.
$$
Let $\gamma=\frac{1}{2}\left(u \cdot u_0\right)>0$ for $u \in O\left(v_0\right)$, noting that $R_{k_l} u_{k_l} \in [h_{k_l}]_{\mu_{k_l}}$. For sufficiently large $l$, it follows that for all $u \in O(v_0)$,
$$
u \cdot u_{k_l}>\gamma,
$$
and
$$
\mu_{k_l}\left(O\left(v_0\right)\right) \geq \mu\left(O\left(v_0\right)\right)>0.
$$
Thus
$$
h_{k_l}(u) \geq R_{k_l}\left(u \cdot u_{k_l}\right)>R_{k_l} \gamma.
$$
Therefore when $l$ is sufficiently large, we obtain
$$
\begin{aligned}
g_{k_l} & =\int_{\mathbb{S}^{n-1}} h_{k_l}^q(u) d \mu_{k_l}(u) \\
& \geq \int_{O\left(v_0\right)} h_{k_l}^q(u) d \mu_{k_l}(u) \\
& >R_{k_l}^q \int_{O\left(v_0\right)} \gamma^q d \mu_{k_l}(u) \\
& \geq R_{k_l}^q \int_{O\left(v_0\right)} \gamma^q d \mu(u),
\end{aligned}
$$
which implies that $g_{k_l} \rightarrow+\infty$ as $l \rightarrow+\infty$. This contradicts (\ref{Eq:q66js}). That is, $\left\{h_k\right\}$ is uniformly bounded on $\mathbb{S}^{n-1}$.

Since $\{h_k\}$ is the support function of $\{[h_k]_\mu\}$, thus $\{[h_k]_\mu\}$ is also uniformly bounded. By the Blaschke Selection Theorem \cite[Theorem 1.8.7]{SRA2014}, there is a subsequence of $\{[h_k]_\mu\}$, also written as $\{[h_k]_\mu\}$, converges to a compact convex set $[h]_\mu$ in $\mathbb{R}^n$. 

The continuity of $\tau_{F,p}$ assures that $\tau_{F,p}([h]_\mu)=1$. 
 From (\ref{Eq:sjgj}), we have 
\begin{align*}
1=\tau_{F,p}([h]_\mu) \leq\frac{p-1}{n(p-1)+p}n^{-\frac{1}{p-1}}\kappa_n^{-\frac{p}{n(p-1)}}V([h]_\mu)^{\frac{n(p-1)+p}{n(p-1)}}.
\end{align*}
Consequently, we have,
\begin{align*} 
V([h]_\mu)\geq\left(\frac{n(p-1)+p}{p-1}n^{\frac{1}{p-1}}\kappa_n^{\frac{p}{n(p-1)}}\right)^{\frac{n(p-1)}{n(p-1)+p}}>0.
\end{align*}
This implies that $[h]_\mu$ is a convex body in $\mathbb{R}^n$.

Thus $[h_k]_\mu\rightarrow [h]_\mu\in\mathscr{K}^n_0$ as $k \rightarrow+\infty$ and $h=h_{[h]_\mu}\geq0$ is the support function of $[h]_\mu$, moreover,
\begin{align*}
\lim _{k \rightarrow+\infty} \lambda_k&=\lim _{k \rightarrow+\infty}\left(\frac{p-1}{n(p-1)+p} \int_{\mathbb{S}^{n-1}} h_k^q(u) d \mu_k(u)\right)\\
&=\frac{p-1}{n(p-1)+p}\int_{\mathbb{S}^{n-1}} h^q(u) d \mu(u)\\&=\lambda \geq 0.
\end{align*}
Thus from (\ref{eq:lxsscg}), (\ref{eq:lx66g}) and (\ref{EQ:LQRSN}), we obtain
$$
d\mu=\lambda dS_{F,p,q}\left([h]_\mu, \cdot\right).
$$
Obviously, $\lambda\neq0$, then $\lambda>0$, let $K=[h]_\mu$, we have
$$
d\mu=\lambda d S_{F,p,q}(K, \cdot),
$$
where $\lambda$ is a positive constant depending on $p$, $q$, $n$ and $\mu$, this yields the desired result.
\end{proof}
\end{sloppypar}

\end{document}